\documentclass[10pt,reqno]{amsart}
\usepackage{amsmath}
\usepackage{amssymb}
\usepackage{amsthm}
\usepackage[usenames]{color}
\usepackage{eepic,epic}
\usepackage{enumitem}   
\usepackage{amsmath}
\usepackage{float}

\newtheorem{thm}{Theorem}[section]
\newtheorem{cor}[thm]{Corollary}
\newtheorem{lem}[thm]{Lemma}
\newtheorem{prop}[thm]{Proposition}

\theoremstyle{definition}

\theoremstyle{remark}
\newtheorem{rem}[thm]{Remark}
\numberwithin{equation}{section}

\newcommand{\R}{\mathbb{R}}

\begin{document}
\title[]
{On critical and supercritical pseudo-relativistic nonlinear Schro\"dinger equations}
\author{Woocheol Choi}
\address{School of Mathematics, Korea Institute for Advanced Study, Seoul 02455, Korea}
\email{wchoi@kias.re.kr}

\author{Younghun Hong}
\address{Department of Mathematics, Yonsei University, Seoul 03722, Korea}
\email{younghun.hong@yonsei.ac.kr}

\author{Jinmyoung Seok}
\address{Department of Mathematics, Kyonggi University, Suwon 16227, Korea}
\email{jmseok@kgu.ac.kr}


\maketitle
\begin{abstract}
In this paper, we investigate existence and non-existence of a nontrivial solution to the
pseudo-relativistic nonlinear Schr\"odinger equation
\begin{equation*}
\left( \sqrt{-c^2\Delta + m^2 c^4}-mc^2\right) u + \mu u = |u|^{p-1}u\quad \textrm{in}~\mathbb{R}^n~(n \geq 2)
\end{equation*}
 involving an $H^{1/2}$-critical/supercritical power-type nonlinearity, i.e., $p \geq \frac{n+1}{n-1}$. 
We prove that in the non-relativistic regime, there exists a nontrivial solution provided that the nonlinearity is $H^{1/2}$-critical/supercritical but it is $H^1$-subcritical.
On the other hand, we also show that there is no nontrivial bounded solution either $(i)$ if the nonlinearity is $H^{1/2}$-critical/supercritical in the ultra-relativistic regime or $(ii)$ if the nonlinearity is $H^1$-critical/supercritical in all cases.
\end{abstract}


\section{Introduction}
We consider the pseudo-relativistic nonlinear Schr\"odinger equation (NLS)
\begin{equation}\label{pseudo-relativistic NLS}
i\partial_t\psi=\left(\sqrt{-c^2\Delta + m^2 c^4}-mc^2\right)\psi - |\psi|^{p-1}\psi,
\end{equation}
where $\psi=\psi(t,x):\mathbb{R}\times\mathbb{R}^n\to\mathbb{C}$ denotes a wave function, $c>0$ is the speed of light and $m>0$ is the particle mass. The non-local operator $\sqrt{-c^2\Delta + m^2 c^4}-mc^2$ is the pseudo-differential operator defined by the symbol $\sqrt{c^2|\xi|^2+m^2c^4}-mc^2$ which arises from Einstein's energy momentum relation $E^2 = |p|^2c^2 +m^2c^4$. 
The equation \eqref{pseudo-relativistic NLS} formally converges to  the non-relativistic NLS 
$$i\partial_t\psi=-\frac{1}{2m}\Delta\psi - |\psi|^{p-1}\psi$$
as $c\to \infty$, because $\sqrt{c^2|\xi|^2 + m^2 c^4}-mc^2\to\frac{1}{2m}|\xi|^2$ as $c\to\infty$. On the other hand, it converges to the nonlinear half-wave equation
$$i\partial_t\psi= c\sqrt{-\Delta}\psi - |\psi|^{p-1}\psi$$
as $m\to 0$. In this sense, the pseudo-relativistic NLS describes the intermediate dynamics between the classical and the relativistic models. The equation \eqref{pseudo-relativistic NLS} may have completely different characters in the non-relativistic regime $(c^2|\xi|^2 \ll m^2c^4)$, where the relativity is taken in account only weakly, from the ultra-relativistic regime $(c^2|\xi|^2 \gg m^2c^4)$, where the relativity is dominant.

The goal of this article is to find criteria for existence and non-existence of a nontrivial standing wave of the form $\psi(t,x)=e^{i\mu t}u(x)$ to the pseudo-relativistic NLS \eqref{pseudo-relativistic NLS}. To this end, we shall focus on the stationary pseudo-relativistic NLS
\begin{equation}\label{eq-main}
( \sqrt{-c^2\Delta + m^2 c^4}-mc^2) u + \mu u = |u|^{p-1}u,
\end{equation}
where $u=u(x):\mathbb{R}^n\to\mathbb{C}$.

When the nonlinearity is $H^{1/2}$-subcritical, i.e., $1 < p < \frac{n+1}{n-1}$, by a standard variational argument, it is shown that the pseudo-relativistic NLS \eqref{eq-main} admits a nontrivial solution for all $m, c, \mu> 0$ (see \cite{CS} and \cite{CiS, CiS2, CN1, CN2, M, TWY} for the related variational results). However, to the best knowledge of the authors, nothing is known in the $H^{1/2}$-critical/supercritical case, i.e., $p\geq\frac{n+1}{n-1}$, because the standard variational approach does not work well in the supercritical setting.

Nevertheless, there is still a hope to construct a nontrivial solution in the critical/supercritical case. To see this, we recall that as $c\to \infty$, the pseudo-relativistic equation \eqref{eq-main} approaches to the non-relativistic equation, that is, the stationary non-relativistic NLS
\begin{equation}\label{non-relativistic eq}
-\frac{1}{2m}\Delta u + \mu u = |u|^{p-1}u.
\end{equation}
As for existence of a nontrivial solution to \eqref{non-relativistic eq}, there is a dichotomy divided at the $H^1$-criticality \cite{BL, S}. Precisely, a positive radially symmetric bounded solution exists in the $H^1$-subcritical case, i.e., $1< p < \frac{n+2}{n-2}$. Moreover, such a solution is known to be unique \cite{Kw}.  However, by Pohozaev's identities, no nontrivial bounded solution exists in the $H^1$-critical/supercritcal case $p \geq \frac{n+2}{n-2}$.  On the other hand, as $m\to 0$, the equation \eqref{eq-main} approaches to the ultra-relativistic equation, namely the stationary nonlinear half-wave equation
\begin{equation}\label{half-wave eq}
c\sqrt{-\Delta}u+\mu u=|u|^{p-1}u.
\end{equation}
Similarly but differently, for \eqref{half-wave eq}, the dichotomy arises at the $H^{1/2}$-criticality. Indeed, a positive radial solution exists in the $H^{1/2}$-subcritical case, i.e., $1 < p < \frac{n+1}{n-1}$,
and its uniqueness is proved by Frank-Lenzmann for $n = 1$ \cite{FLe} and Frank-Lenzmann-Silvestre for $n\geq 2$ \cite{FLS}, provided that it is also a ground state. However, by Pohozaev's identities again, a bounded nontrivial solution does not exist in the $H^{1/2}$-critical/supercritical case $p\geq\frac{n+1}{n-1}$.
\floatplacement{table}{h}
\begin{table}[]
\centering
\label{known results}
\begin{tabular}{lllll}
\cline{1-4}
\multicolumn{1}{|c|}{} & \multicolumn{1}{c|}{$1 < p < \frac{n+1}{n-1}$}         & \multicolumn{1}{c|}{$\frac{n+1}{n-1} \leq p < \frac{n+2}{n-2}$} & \multicolumn{1}{c|}{$\frac{n+2}{n-2}\leq p<\infty$}             &  \\ \cline{1-4}
\multicolumn{1}{|l|}{non-relativistic NLS \eqref{non-relativistic eq}} & \multicolumn{2}{c|}{existence}                          & \multicolumn{1}{c|}{non-existence} &  \\ \cline{1-4}
\multicolumn{1}{|l|}{half wave equation \eqref{half-wave eq}} & \multicolumn{1}{c|}{existence} & \multicolumn{2}{c|}{non-existence}                          &  \\ \cline{1-4}
                        &                                &                        &                                    & 
\end{tabular}
\caption{Existence and non-existence of a non-trivial solution to \eqref{non-relativistic eq}/\eqref{half-wave eq}}
\end{table}

The above observation suggests a possibility that existence of a nontrivial solution to \eqref{eq-main} can be determined by the criticality of the equation as well as by the parameters $m, c, \mu> 0$. More specifically, from the results in Table \ref{known results} and the connections among the three equations via the formal limits, it is natural to guess that when $\frac{n+1}{n-1}\leq p<\frac{n+2}{n-2}$, a nontrivial solution exists in the non-relativistic regime $c\gg1$, but it does not in the ultra-relativistic regime $m\ll1$. No existence is expected when $p\geq\frac{n+2}{n-2}$.

The first theorem of this paper proves non-existence of a nontrivial solution to \eqref{eq-main}, which fits into Table \ref{known results}.

\begin{thm}[Non-existence]\label{thm-3}
Let $n \geq 2$. Suppose that
$$p \geq \frac{n+1}{n-1}  \quad \text{and} \quad mc^2 \leq \mu$$ or that 
$$p \geq \frac{n+2}{n-2}.$$
Then, there is no nontrivial solution to \eqref{eq-main} in $H^{1/2}(\R^n) \cap L^\infty(\R^n) $.
\end{thm}

We show Theorem \ref{thm-3} by exploiting the Pohozaev-type identities on the extension problem of \eqref{eq-main} to the upper half-plane. We note that if $\frac{n+1}{n-1}\leq p<\frac{n+2}{n-2}$, this approach does not work when $mc^2>\mu$. 

The next theorem, which is the main contribution of this article, provides an affirmative answer for the existence part.

\begin{thm}[Existence]\label{existence}
Let $n\geq 2$. Suppose that
\begin{align*}
\frac{n+1}{n-1} \leq p<\frac{n+2}{n-2}.
\end{align*}
Then, there exists $\kappa_0\geq 1$ such that if
$${mc^2}\geq \kappa_0{\mu},$$
then the pseudo-relativistic NLS \eqref{eq-main} has a nontrivial solution $u_c \in H^1_r(\R^n) \cap L^\infty(\R^n)$. 
\end{thm}

Even though a lot of works have been devoted to the pseudo-relativistic NLS \eqref{eq-main}, to the best knowledge of the authors, Theorem \ref{existence} is the first result in the literature, proving existence of its non-trivial solution in the supercritical setting. Another important remark is that only the quantity $\frac{mc^2}{\mu}$, not all three independent parameters $m,c,\mu>0$, determines the regime of the equation concerning existence and non-existence of a nontrivial solution. 

Due to supercriticality of the problem, it is difficult to apply the standard variational method. Indeed, the action functional (see \eqref{action functional}) is indefinite in this case. Cutting off the nonlinearity is also not suitable in this setting. 
Even in the $H^{1/2}$-critical case, due to the lack of $L^2$-convergence of Palais-Smale sequences, it looks impossible to apply a well-known method that characterizes energy sub-levels such that the Palais-Smale condition holds and subsequently tests a family of extremal functions of the Sobolev inequality.

To overcome the aforementioned difficulty, we employ the following non-variational approach, combined with the uniform $L^q$-estimates for the pseudo-differential operator $\sqrt{-c^2\Delta +m^2c^4}-mc^2$. First, by some algebraic manipulation for notational simplicity, we reduce to the case $m=\frac{1}{2}$ and $\mu=1$, and consider 
\begin{equation}\label{eq-3-9}
P_c (D) u= |u|^{p-1} u,
\end{equation}
where
$$P_c (D)=\left(\sqrt{ -c^2 \Delta + \frac{c^4}{4}} -\frac{c^2}{2}\right)+1$$
(see Section \ref{algebra}).  Next, we choose a ground state $u_\infty$ to the non-relativistic limit equation
\begin{equation}\label{eq-3-9 limit}
-\Delta u+u=|u|^{p-1}u.
\end{equation}
Considering the $H^{1/2}$-supercritical pseudo-relativistic equation \eqref{eq-3-9} as a perturbation of the $H^1$-subcritical non-relativistic equation \eqref{eq-3-9 limit}, we formulate an equation for the perturbation $w$ from the ground state $u_\infty$ (see \eqref{eq w''}). 
Then, we establish existence of a solution $w$ via the contraction mapping principle (see Theorem \ref{existence and uniqueness}). 
The main advantage of this approach is that we may take the full advantage of extra properties of the ground state $u_\infty$, including its smoothness and decay. 
In particular, the non-degeneracy of the linearized operator $\mathcal{L}_\infty$ about the ground state $u_\infty$ (see \eqref{non-degeneracy}) plays a crucial role in this procedure. This kind of perturbation argument has been employed previously in the literature for other problems. For example, we refer to \cite{CGG, G, O} for the nonlinear Dirac equation and to \cite{MMP} for the nonlinear Schr\"odinger equation with slightly supercritical nonlinearity.

Another new ingredient of our analysis is to use the $L^q$-estimates for the pseudo-differential operator $P_c (D)$ based on the symbolic analysis. Indeed, for our contraction mapping argument, it is important to find a uniform (in $c\geq2$) boundedness of the inverse operator $P_c (D)^{-1}: L^q \rightarrow W^{1,q}$. In the special case $q=2$, such an estimate immediately follows from a simple pointwise bound on the symbol \cite[Lemma 4.3]{CHS}. In this paper, we obtain the following extended inequalities covering all exponents $1<q<\infty$.

\begin{thm}[Norm comparability]\label{norm comparability}
For $1<q<\infty$, there exists a constant $C_{q,n} >0$ such that for $c\geq 2$, 
\begin{equation*}
C_{q,n}^{-1} \| f\|_{W^{1,q}} \leq \|P_c (D) f\|_{L^q} \leq C_{q,n}\|f\|_{W^{2,q}}
\end{equation*}
\end{thm}

\begin{rem}
Since $P_c(D)$ is a first-order elliptic operator, it is obvious that
$$C_{q,n, c}^{-1}\|f\|_{W^{1,q}}\leq \|P_c(D) f\|_{L^q}\leq C_{q,n,c}\|f\|_{W^{1,q}},$$
with $C_{q,n,c}>0$ depending on $c \geq 2$. Contrary to these trivial inequalities, Theorem \ref{norm comparability} provides upper and lower bounds uniformly in $c\geq 2$. Note that the upper bound in Theorem \ref{norm comparability} is optimal, because $\|P_c(D) f\|_{L^q}\to \|(-\Delta+1) f\|_{L^q}$ as $c\to \infty$ for $f\in C_c^\infty$.
\end{rem}

We prove Theorem \ref{norm comparability} by the H\"ormander-Mikhlin theorem. For this aim, we carefully estimate the derivatives of the associated symbols. We here also prove that for sufficiently large $c\geq 1$, the inverse of the pseudo-relativistic operator is close to that of the non-relativistic operator as operators acting on $L^q$ (see Theorem \ref{thm difference}). Theorem \ref{norm comparability} and \ref{thm difference} are employed to obtain existence in the full $H^1$-subcritical range. Indeed, without these extended inequalities, only a narrow range of nonlinearities, $1<p<\frac{n}{n-2}$, is covered (see Remark \ref{Why Hormander-Mikhlin}). 



The rest of this paper is organized as follows. In Section \ref{sec: symbol calculus}, we establish several mapping properties for the pseudo-differential operator $P_c (D)$. Given those properties, in Section \ref{sec: existence result}, we prove the existence result (Theorem \ref{existence}). Section \ref{sec-4} is devoted to establish non-existence (Theorem \ref{thm-3}). 
Finally, in Section \ref{sec-5}, we discuss some properties of the solution constructed previously and propose an open question related to these properties.    

\subsection{Acknowledgement}. This research of the second author was supported by Basic Science Research Program through the National Research Foundation of Korea(NRF) funded by the Ministry of Education (NRF-2017R1C1B1008215).

\section{Symbol calculus for the pseudo-relativistic Schr\"odinger operator}\label{sec: symbol calculus}

Given a symbol $\mathfrak{m}:\mathbb{R}^d\to\mathbb{R}$, the associated Fourier multiplier operator $\mathfrak{m}(D)$ is defined by 
$$\widehat{\mathfrak{m}(D) f} (\xi) = \mathfrak{m} (\xi)\hat{f} (\xi).$$
We introduce the pseudo-differential operator $P_c(D)$ (or $P_\infty(D)$, respectively) as the Fourier multiplier with the symbol 
\begin{equation}\label{def: P_c, P_infty}
P_c(\xi):=\left(\sqrt{c^2|\xi|^2+\frac{c^4}{4}}-\frac{c^2}{2}\right)+1\quad\left(P_\infty(\xi):=|\xi|^2+1, \textup{respectively}\right).
\end{equation}
The purpose of this section is to provide the connection between these two operators. Precisely, we show that as inverse operators, $P_c(D)$ converges to $P_\infty(D)$ as $c\to \infty$ (Theorem \ref{thm difference} below). Here, we also prove the norm comparability (Theorem \ref{norm comparability}).

\begin{thm}[Difference between the inverses of two operators]\label{thm difference}
\mbox{~}
\begin{enumerate}
\item For $1<q<\infty$, there exists a constant $C_{q,n} >0$ such that
\begin{equation}\label{symbol estimate 1}
\left\|\left(\frac{1}{P_\infty(D)}-\frac{1}{P_c(D)}\right)f\right\|_{L^q}\leq \frac{C_{q,n}}{c^2}\|f\|_{L^q}\quad \forall\, c \in [2,\infty).
\end{equation}
\item For $1<q<\infty$, there exists a constant $C_{q,n} >0$ such that
\begin{equation}\label{symbol estimate 2}
\left\|\left(\frac{1}{P_\infty(D)}-\frac{1}{P_c(D)}\right)f\right\|_{L^q}\leq \frac{C_{q,n}}{c}\left\|\frac{1}{P_\infty(D)^{\frac{1}{2}}}f\right\|_{L^q}\quad\forall\, c \in [2,\infty).
\end{equation}
\end{enumerate}
\end{thm}

For the proof, we recall the H\"ormander-Mikhlin multiplier theorem (see \cite{Gr} for instance).
\begin{thm}[H\"ormander-Mikhlin]\label{thm-2-2} Suppose that $\mathfrak{m} : \mathbb{R}^n \rightarrow \mathbb{R}$ satisfies 
\begin{equation*}
|\nabla^{\alpha} \mathfrak{m} (\xi)| \leq \frac{B_{\alpha}}{|\xi|^{|\alpha|}}\quad \forall \xi \in \mathbb{R}^n \setminus \{0\}
\end{equation*}
for all multi-indices $\alpha \in (\mathbb{Z}_{\geq 0})^n$ such that $0\leq |\alpha| \leq \frac{n}{2} +1$, where $\mathbb{Z}_{\geq 0}=\{0,1,2,3\cdots \}$. Then for any $1<q<\infty$, there exists a constant $C_{q,n} >0$ such that 
\begin{equation*}
\|\mathfrak{m}(D) f\|_{L^q} \leq C_{q,n} \left( \sup_{0 \leq |\alpha| \leq \frac{n}{2} + 1} B_{\alpha}\right) \|f\|_{L^q}.
\end{equation*}
\end{thm}

By the H\"ormander-Mikhlin multiplier theorem, the proofs of Theorem \ref{thm difference} and Theorem \ref{norm comparability} are reduced to those of the following bounds on the derivatives of the symbols.
\begin{prop}\label{symbol fraction}\mbox{~}
\begin{enumerate}
\item 
For any multi-index $\alpha=(\alpha_1,\cdots, \alpha_n)\in (\mathbb{Z}_{\geq 0})^n$, there is a constant $C_{\alpha}>0$ such that for all $c \in [2,\infty)$,
\begin{equation}\label{eq-7-2}\left|\nabla_\xi^\alpha \left(\frac{1}{P_\infty(\xi)}-\frac{1}{P_c(\xi)}\right)\right|\leq\frac{C_\alpha}{|\xi|^{|\alpha|}}\min\left\{\frac{1}{c^2}, \frac{1}{c(|\xi|^2+1)^{\frac{1}{2}}}\right\}.
\end{equation}
\item 
For any multi-index $\alpha=(\alpha_1,\cdots, \alpha_n)\in (\mathbb{Z}_{\geq 0})^n$, there is a constant $C_{\alpha}>0$ such that for all $c \in [2,\infty)$,
\begin{equation}\label{eq-7-1}\left|\nabla_\xi^\alpha \left(\frac{P_c(\xi)}{P_\infty(\xi)}\right)\right|\leq\frac{C_\alpha}{|\xi|^{|\alpha|}}.
\end{equation}
\end{enumerate}
\end{prop}

\begin{proof}[Proof of Theorem \ref{thm difference} assuming Proposition \ref{symbol fraction}]
For any multi-index $\alpha \in (\mathbb{Z}_{\geq 0})^n$, it follows from estimate \eqref{eq-7-2} that 
$$\left|\nabla_\xi^\alpha\left(\frac{1}{P_\infty(\xi)}-\frac{1}{P_c(\xi)}\right)\right|\leq\frac{C_\alpha}{c^2|\xi|^{|\alpha|}}$$ and
\begin{align*}
\left|\nabla_\xi^\alpha\left(\left(\frac{1}{P_\infty(\xi)}-\frac{1}{P_c(\xi)}\right)P_\infty(\xi)^{\frac{1}{2}}\right)\right|&\leq\sum_{\alpha_1+\alpha_2=\alpha} \left|\nabla_\xi^{\alpha_1}\left(\frac{1}{P_\infty(\xi)}-\frac{1}{P_c(\xi)}\right)\right|\left|\nabla_\xi^{\alpha_2}\left(P_\infty(\xi)^{\frac{1}{2}}\right)\right|\\
&\leq \sum_{\alpha_1+\alpha_2=\alpha}\frac{C_{\alpha_1}}{c|\xi|^{|\alpha_1|+1}}\frac{C_{\alpha_2}}{|\xi|^{|\alpha_2|-1}}= \sum_{\alpha_1+\alpha_2=\alpha}\frac{C_{\alpha_1}C_{\alpha_2}}{c|\xi|^{|\alpha|}}\\
&\leq\frac{C_\alpha}{c|\xi|^{|\alpha|}}.
\end{align*}
Thus, Theorem \ref{thm difference} follows from Theorem \ref{thm-2-2}.
\end{proof}

\begin{proof}[Proof of Theorem \ref{norm comparability} assuming Proposition \ref{symbol fraction}]
By the triangle inequality and \eqref{symbol estimate 2}, we prove that 
\begin{align*}
\left\|\frac{1}{P_c(D)}f\right\|_{L^q}&\leq \left\|\frac{1}{P_\infty(D)}f\right\|_{L^q}+\left\|\left(\frac{1}{P_\infty(D)}-\frac{1}{P_c(D)}\right)f\right\|_{L^q}\\
&\leq \left\|\frac{1}{P_\infty(D)}f\right\|_{L^q}+\frac{C_q}{c}\left\|\frac{1}{P_\infty(D)^{\frac{1}{2}}}f\right\|_{L^q}\\
&\leq C_{q}\left\|f\right\|_{W^{-1,q}}.
\end{align*}
By inserting $f \rightarrow {P_c (D)}\sqrt{P_{\infty}(D)} f$, we prove the first inequality. The second inequality follows from \eqref{eq-7-1} and Theorem \ref{thm-2-2}.
\end{proof}
In the rest of this section, we prove Proposition \ref{symbol fraction}. For this aim, we first observe that the pseudo-relativistic symbol $P_c(\xi)$ is comparable with the non-relativistic symbol $P_\infty(\xi)$ inside a large ball, while it is like the half-wave symbol $c|\xi|+1$ outside. 

\begin{lem}[Pointwise comparability on the pseudo-relativistic symbol]\label{P_c comparison}
$$
\left\{ \begin{aligned}
\tfrac{|\xi|^2+1}{2}&\leq P_c(\xi)\leq |\xi|^2+1&&\textup{ if }|\xi|\leq\tfrac{\sqrt{3}c}{2},\\
\tfrac{c|\xi|+1}{2}&\leq P_c(\xi)\leq c|\xi|+1&&\textup{ if }|\xi|\geq\tfrac{\sqrt{3}c}{2}.
\end{aligned}\right.
$$
\end{lem}

\begin{proof}
Suppose that $|\xi|\leq\frac{\sqrt{3}c}{2}$. Then, $P_c(\xi)=\frac{c^2}{2}f(\frac{4|\xi|^2}{c^2})+1$, where $f(t)=\sqrt{1+t}-1$. Since $f(0)=0$ and $\frac{1}{4}\leq f'(t)=\frac{1}{2\sqrt{1+t}}\leq\frac{1}{2}$ on $[0,3]$, by Taylor's theorem,  we have
$$\frac{|\xi|^2}{2}=\frac{c^2}{2}\cdot\frac{1}{4}\cdot\frac{4|\xi|^2}{c^2}\leq P_c(\xi)-1\leq \frac{c^2}{2}\cdot\frac{1}{4}\cdot\frac{4|\xi|^2}{c^2}=|\xi|^2.$$
On the other hand, if $|\xi|\geq\frac{\sqrt{3}c}{2}$, then
$$P_c(\xi)-1=c|\xi|\left(\sqrt{1+\frac{c^2}{4|\xi|^2}}-\frac{c}{2|\xi|}\right)=\frac{c|\xi|}{\sqrt{1+\frac{c^2}{4|\xi|^2}}+\frac{c}{2|\xi|}}$$
obeys the desired upper and lower bound, because $\frac{c|\xi|}{\sqrt{3}}\leq P_c(\xi)-1\leq c|\xi|$.
\end{proof}

\begin{rem}\label{P_c comparison remark}
Indeed, the inequality $P_c(\xi)\leq |\xi|^2+1$ holds for all $\xi$, since in the proof of Lemma \ref{P_c comparison}, we can use $|f'(t)|\leq\frac{1}{2}$ for all $t\geq 0$.
\end{rem}
Next, we show that the the pseudo-relativistic symbol $P_c(\xi)$ approximates the non-relativistic symbol $P_\infty(\xi)$ near the origin.
 
\begin{lem}[Pointwise estimate on the difference between the two symbols]\label{P_c comparison'}
$$|P_c(\xi)-P_\infty(\xi)|\leq\frac{|\xi|^4}{c^2}.$$
\end{lem}

\begin{proof}
Let $f(t)$ be the function given in the proof of Lemma \ref{P_c comparison'}. Then, by Taylor's theorem together with $f''(t)=-\frac{1}{4}(1+t)^{-\frac{3}{2}}$, we have 
$$|P_c(\xi)-P_\infty(\xi)|=\left|\frac{c^2}{2}\left(f(0)+f'(0)\tfrac{4|\xi|^2}{c^2}+\frac{1}{2}f''(t_*)(\tfrac{4|\xi|^2}{c^2})^2\right)-|\xi|^2\right|\leq\frac{|\xi|^4}{c^2}$$
for some $t_*\in (0,\tfrac{4|\xi|^2}{c^2})$.
\end{proof}
Now we are ready to prove Proposition \ref{symbol fraction}.
\begin{proof}[Proof of Proposition \ref{symbol fraction}]
We \eqref{eq-7-2} and \eqref{eq-7-1} separately.

\medskip 

\noindent \textbf{Proof of \eqref{eq-7-2}.} We denote 
$$a(\xi):=\frac{1}{P_\infty(\xi)}-\frac{1}{P_c(\xi)}=\frac{P_c(\xi)-P_\infty(\xi)}{P_c(\xi) P_\infty(\xi)}.$$
First, we find a structure of the derivatives of the symbol $a$. Precisely, we claim that $\nabla_\xi^\alpha a(\xi)$ is the sum of products of the following factors,
\begin{equation}\label{factors}
a(\xi),\ \frac{1}{P_c(\xi)^{\ell_1}},\ \frac{1}{P_\infty(\xi)^{\ell_2}},\ \frac{1}{(c^2|\xi|^2+\frac{c^4}{4})^{\frac{\ell_3}{2}}},\  \textup{a polynomial of } c, \xi_1,\cdots, \xi_n,
\end{equation}
where $\ell_1,\ell_2,\ell_3\in\mathbb{Z}_{\geq 0}$. The claim is obviously true when $\alpha=0$. Suppose that the claim holds for some multi-index $\alpha$, and consider its derivative. By the induction hypothesis, $(\nabla_{\xi_j}\nabla_{\xi}^\alpha a)(\xi)$ is the sum of the derivatives of the products. By the Leibniz rule, the derivative of each product term is the sum of the derivative of one factor in \eqref{factors} times the product of others in \eqref{factors}. Thus, it suffices to show that the derivative of any of \eqref{factors} is again a product of terms in \eqref{factors}. Of course,  the derivative of a polynomial of $c, \xi_1,\cdots, \xi_n$ is also a polynomial of the same variables. By direct calculations, we observe that 
\begin{equation}\label{derivative of factors}
\begin{aligned}
\nabla_{\xi_j}\left(P_c(\xi)-P_\infty(\xi)\right)&=-\frac{2\xi_jP_c(\xi)}{(c^2|\xi|^2+\frac{c^4}{4})^{\frac{1}{2}}},\\
\nabla_{\xi_j}\left(\frac{1}{P_c(\xi)^{\ell_1}}\right)&=-\frac{\ell_1c^2\xi_j}{(c^2|\xi|^2+\frac{c^4}{4})^{\frac{1}{2}}P_c(\xi)}\left(\frac{1}{P_c(\xi)^{\ell_1}}\right),\\
\nabla_{\xi_j}\left(\frac{1}{P_\infty(\xi)^{\ell_2}}\right)&=-\frac{2\ell_2\xi_j}{P_\infty(\xi)}\left(\frac{1}{P_\infty(\xi)^{\ell_2}}\right),\\
\nabla_{\xi_j}\left(\frac{1}{(c^2|\xi|^2+\frac{c^4}{4})^{\frac{\ell_3}{2}}}\right)&=-\frac{\ell_3\xi_j}{|\xi|^2+\frac{c^2}{4}}\left(\frac{1}{(c^2|\xi|^2+\frac{c^4}{4})^{\frac{\ell_3}{2}}}\right).
\end{aligned}
\end{equation}
Hence, by the Leibniz rule, 
\begin{equation}\label{derivative of a}
\begin{aligned}
(\nabla_{\xi_j}a)(\xi)&=\nabla_{\xi_j}\left(\frac{P_c(\xi)-P_\infty(\xi)}{P_c(\xi) P_\infty(\xi)}\right)\\
&=-\frac{2\xi_j}{P_\infty(\xi)(c^2|\xi|^2+\frac{c^4}{4})^{\frac{1}{2}}}-\frac{\ell_1c^2\xi_j}{(c^2|\xi|^2+\frac{c^4}{4})^{\frac{1}{2}}P_c(\xi)}a(\xi)-\frac{2\ell_2\xi_j}{P_\infty(\xi)}a(\xi).
\end{aligned}
\end{equation}
From this, we conclude that when a derivative hits a factor in \eqref{factors}, it does not make a new type of factors other than \eqref{factors}. Thus, the claim is proved.

We also note from \eqref{derivative of factors} that when $\frac{1}{P_c(\xi)^{\ell_1}}$, $\frac{1}{P_\infty(\xi)^{\ell_2}}$, $\frac{1}{(c^2|\xi|^2+\frac{c^4}{4})^{\frac{\ell_3}{2}}}$ are differentiated, extra factors are produced. Moreover, these extra factors are all bounded by $\frac{C}{|\xi|}$. Indeed, by Lemma \ref{P_c comparison}, 
\begin{equation}\label{extra term estimate}
\begin{aligned}
\left|\frac{\ell_1c^2\xi_j}{(c^2|\xi|^2+\frac{c^4}{4})^{\frac{1}{2}}P_c(\xi)}\right|&\leq\left\{ \begin{aligned}
&\frac{\ell_1c^2|\xi|}{\frac{c^2}{2}\frac{|\xi|^2}{2}}=\frac{4\ell_1}{|\xi|}&&\textup{ if }|\xi|\leq\tfrac{\sqrt{3}c}{2},\\
&\frac{\ell_1c^2|\xi|}{c|\xi|\frac{c|\xi|}{2}}=\frac{2\ell_1}{|\xi|}&&\textup{ if }|\xi|\geq\tfrac{\sqrt{3}c}{2},
\end{aligned}\right.\\
\left|\frac{2\ell_2\xi_j}{P_\infty(\xi)}\right|&\leq\frac{2\ell_2|\xi|}{|\xi|^2+1}\leq \frac{2\ell_2}{|\xi|},\\
\left|\frac{\ell_3\xi_j}{|\xi|^2+\frac{c^2}{4}}\right|&\leq\frac{\ell_3|\xi|}{|\xi|^2+1}\leq  \frac{\ell_3}{|\xi|}.
\end{aligned}
\end{equation}

We now prove the proposition by induction. For the zeroth induction step, i.e., $\alpha=0$, using Lemma \ref{P_c comparison}, we show that if $|\xi|\leq\frac{\sqrt{3}c}{2}$, then
$$|a(\xi)|\leq\frac{\frac{|\xi|^4}{c^2}}{\frac{|\xi|^2+1}{2}(|\xi|^2+1)}\leq\min\left\{\frac{2}{c^2},\frac{2|\xi|}{c^2(|\xi|^2+1)^{\frac{1}{2}}}\right\}\leq 2\min\left\{\frac{1}{c^2},\frac{1}{c(|\xi|^2+1)^{\frac{1}{2}}}\right\}.$$
On the other hand, if $|\xi|\geq\frac{\sqrt{3}c}{2}$, then by Lemma \ref{P_c comparison} again, 
\begin{align*}
|a(\xi)|&\leq\frac{1}{P_\infty(\xi)}+\frac{1}{P_c(\xi)}\leq\frac{1}{|\xi|^2+1}+\frac{2}{c|\xi|+1}\\
&\leq\min\left\{\frac{4}{3c^2}+\frac{4}{\sqrt{3}c^2},\frac{2}{\sqrt{3}c(|\xi|^2+1)^{\frac{1}{2}}}+\frac{2\sqrt{2}}{c(|\xi|^2+1)^{\frac{1}{2}}}\right\}\\
&\leq 4\min\left\{\frac{1}{c^2},\frac{1}{c(|\xi|^2+1)^{\frac{1}{2}}}\right\}.
\end{align*}
For the first induction step, i.e., $|\alpha|=1$, we consider the sum \eqref{derivative of a}. By a trivial inequality, the first term in \eqref{derivative of a} is bounded by   
$$\frac{4}{|\xi|} \min\left\{\frac{1}{c^2},\frac{1}{c(|\xi|^2+1)^{\frac{1}{2}}}\right\}.$$
Moreover, it follows from \eqref{extra term estimate} and the zeroth induction step that the second and the last terms in $\eqref{derivative of a}$ also obeys the same bound. Collecting all, we complete the proof of the first induction step.

For induction, we assume that each product in the sum for $(\nabla_\xi^\alpha a)(\xi)$ is bounded by
$$\frac{C_\alpha}{|\xi|^{|\alpha|}}\min\left\{\frac{1}{c^2}, \frac{1}{c(|\xi|^2+1)^{\frac{1}{2}}}\right\}.$$
Then, it suffices to show that each term in the sum for $(\nabla_{\xi_j}\nabla_\xi^\alpha a)(\xi)$ satisfies the desired bound. Indeed, all these terms are obtained by differentiating the product terms in the previous step. However, as mentioned previously, when a product is differentiated, the derivative lands on either a polynomial factor or other types of factors in \eqref{factors}. When a polynomial is differentiated, its degree is reduced by one. Otherwise, an extra factor is generated (see \eqref{derivative of factors}) and such an extra factor is bounded by $\frac{C}{|\xi|}$ (see \eqref{extra term estimate}). Thus, summing up all bounds, we prove the proposition.

\medskip

\noindent \textbf{Proof of \eqref{eq-7-1}. }The proof is very similar to that of estimate \eqref{eq-7-2}, so we only give a sketch of it. First, by Remark \ref{P_c comparison remark}, $|\frac{P_c(\xi)}{P_\infty(\xi)}|\leq 1$. Next, we prove the first derivative
$$\nabla_{\xi_j}\left(\frac{P_c(\xi)}{P_\infty(\xi)}\right)=\frac{1}{P_\infty(\xi)}\cdot\frac{c^2\xi_j}{(c^2|\xi|^2+\frac{c^4}{4})^{\frac{1}{2}}}-\frac{2\xi_j}{P_\infty(\xi)}\frac{P_c(\xi)}{P_\infty(\xi)}$$
is bounded by
$$\frac{2|\xi|}{P_\infty(\xi)}+\frac{2|\xi|}{P_\infty(\xi)}\cdot 1\leq \frac{4}{|\xi|}.$$

By the induction argument in the proof of estimate \eqref{eq-7-2}, one can show that $\nabla_{\xi}^\alpha(\frac{P_c(\xi)}{P_\infty(\xi)})$ is the sum of products of $\frac{P_c(\xi)}{P_\infty(\xi)}$, $\frac{1}{P_\infty(\xi)^{\ell_1}}$, $\frac{1}{(c^2|\xi|^2+\frac{c^4}{4})^{\frac{\ell_2}{2}}}$ and a polynomial of $c, \xi_1,\cdots, \xi_n$. For induction, we assume that each product in the sum for the expansion of $\nabla_{\xi}^\alpha(\frac{P_c(\xi)}{P_\infty(\xi)})$ is bounded by $\frac{C_\alpha}{|\xi|^{|\alpha|}}$. If we differentiate each product, then differentiation produces an extra factor keeping the same structure. Moreover, all possible extra factors are bounded by $\frac{C}{|\xi|}$ (see \eqref{derivative of factors} and \eqref{extra term estimate}). Therefore, the derivative of each product is bounded by $\frac{C_\alpha'}{|\xi|^{|\alpha|+1}}$. Then, summing all the bounds, we obtain the desired bound for the derivative of $\nabla_{\xi}^\alpha(\frac{P_c(\xi)}{P_\infty(\xi)})$.
\end{proof}

\section{Existence result}\label{sec: existence result}

This section is devoted to our main existence theorem whose proof will be divided into several steps. First, in \S\ref{algebra}, by algebraic manipulation, we simplify to the case $m=\frac{1}{2}$ and $\mu=1$. In \S\ref{reformulation}, we reformulate  the pseudo-relativistic Schr\"odinger equation \eqref{eq-main}  as an equation for the perturbation from the non-relativistic ground state (see \eqref{eq w''}). The goal is then to construct a solution to the equation by a standard contraction mapping argument. To that end, we prove several key estimates for contraction in \S\ref{sec: invertibility}-\S\ref{sec: nonlinear estimates}. After being prepared, in \S\ref{sec: construction of w}, we establish existence and uniqueness of a solution to the reformulated equation. Finally, in \S\ref{sec: properties}, we complete the proof of Theorem \ref{existence}.

\subsection{Reduction to the simple case}\label{algebra}

To begin with, we observe that if $v_c$ is a non-trivial solution to the pseudo-relativistic NLS with $m=\frac{1}{2}$ and $\mu=1$, i.e., 
\begin{equation}\label{reduced eq-main}
P_c(D) u=|u|^{p-1}u,
\end{equation}
where
$$P_c(D):=\left(\sqrt{-c^2\Delta+\frac{c^4}{4}}-\frac{c^2}{2}\right)+1,$$
then $u_c(x)=\mu^{\frac{1}{p-1}}v_c(\sqrt{2m\mu} x)$ solves
$$\left(\sqrt{-\tilde{c}^2\Delta+m^2\tilde{c}^4}-m\tilde{c}^2\right)u+\mu u=|u|^{p-1}u,$$
where $\tilde{c}=c\sqrt{\frac{\mu}{2m}}$. Thus, we may restrict ourselves to the case $m=\frac{1}{2}$ and $\mu=1$.

\subsection{Setup for contraction}\label{reformulation}

We aim to find a nontrivial solution to \eqref{reduced eq-main} by employing a perturbation argument. Throughout this section, we assume that $1<p<\infty$ if $n=1,2$, and that $1<p<\frac{n+2}{n-2}$ if $n\geq 3$.  

Let $u_{\infty} \in H_r^1$ be a ground state to the non-relativistic Schr\"odinger equation
$$P_\infty(D) u=|u|^{p-1}u,$$
which is known to be positive and unique. Hoping to find a radially symmetric real-valued solution $u_c$ to the pseudo-relativistic equation \eqref{reduced eq-main} close to the non-relativistic ground state $u_\infty$, we write the equation for the difference
$$w:=u_c-u_\infty:\mathbb{R}^n\to\mathbb{R},$$
that is,
\begin{align*}
P_c(D)w&=P_c(D)u_c-P_c(D)u_\infty\\
&=(P_\infty(D)-P_c(D))u_\infty+P_c(D)u_c-P_\infty(D)u_\infty\\\
&=(P_\infty(D)-P_c(D))u_\infty+\Big\{|u_\infty+w|^{p-1}(u_\infty+w)-u_\infty^p\Big\}.
\end{align*}
Then, subtracting the linear component $p u_\infty^{p-1}w$ from the both sides, we get
$$\mathcal{L}_{c;\infty}w=(P_\infty(D)-P_c(D))u_\infty+\mathcal{Q}(w),$$
where
$$\mathcal{L}_{c;\infty}:=P_c(D)-pu_\infty^{p-1}$$
and
\begin{equation}\label{nonlinearity}
\mathcal{Q}(w):=|u_\infty+w|^{p-1}(u_\infty+w)-u_\infty^p-pu_\infty^{p-1}w.
\end{equation}
Finally, using that the operator $\mathcal{L}_{c;\infty}$ is invertible (see Proposition \ref{invertibility} below), which is the key ingredient in our analysis, we derive the equation
\begin{equation}\label{eq w''}
w=\mathcal{R}_c+(\mathcal{L}_{c;\infty})^{-1}\mathcal{Q}(w),
\end{equation}
where
$$\mathcal{R}_c:=(\mathcal{L}_{c;\infty})^{-1}(P_\infty(D)-P_c(D))u_\infty.$$

We now wish to construct a radially symmetric real-valued solution $w$ for the equation \eqref{eq w''} via the standard contraction mapping argument, assuming that $c\geq 1$ is large enough. Precisely, we aim to show that the nonlinear map
\begin{equation}\label{contraction mapping}
\Phi_c(w):=\mathcal{R}_c+(\mathcal{L}_{c;\infty})^{-1}\mathcal{Q}(w)
\end{equation}
is contractive on a small ball in the Sobolev space $H_r^1\cap W^{1,q}$ of radially symmetric functions so that there is a unique solution $u_c$ to \eqref{reduced eq-main} in a small neighborhood of $u_\infty$. It should be noted that the reformulated equation \eqref{eq w''} is well-suited for our purpose. Indeed, the first term $\mathcal{R}_c$ is small for large $c$, because the ground state $u_\infty$ is a regular function and the symbol $|\xi|^2+1-P_c(\xi)$ is asymptotically $O(\frac{|\xi|^4}{c^2})$ as $c\to\infty$. Moreover, if $w$ is small, then the super-linear nonlinear term $(\mathcal{L}_{c;\infty})^{-1}\mathcal{Q}(w)$ is even smaller. Therefore, it is natural to expect that $\Phi_c$ maps a small ball to itself, and it is contractive on the set. These will be justified rigorously in the next subsections.

\subsection{Invertibility of $\mathcal{L}_{c;\infty}$}\label{sec: invertibility}
The following proposition asserts that the differential operator $\mathcal{L}_{c;\infty}$ is invertible, and moreover its inverse gains one derivative.

\begin{prop}[Invertibility and smoothing property of $\mathcal{L}_{c;\infty}$]\label{invertibility}
Let $2\leq q<\infty$. Then, there exists $c_0>0$ such that if $c \geq c_0$, then $\mathcal{L}_{c;\infty}: H_r^1\cap W^{1,q}\to L_r^2\cap L^q$ is invertible. Moreover, its inverse is uniformly bounded, i.e.,
$$\sup_{c\geq c_0}\|(\mathcal{L}_{c;\infty})^{-1}\|_{\mathcal{L}(L_r^2\cap L^q; H_r^1\cap W^{1,q})}<\infty,$$
where $\|\cdot\|_{\mathcal{L}(X;Y)}$ is the operator norm from the Banach space $X$ to the Banach space $Y$.
\end{prop}

The proof of the proposition heavily relies on the non-degeneracy of the linearized operator
$$\mathcal{L}_\infty:=-\Delta+1-pu_\infty^{p-1}=P_\infty(D)-pu_\infty^{p-1}: H^2\to L^2$$
about the non-relativistic ground state $u_\infty$ for radially symmetric functions, that is,
\begin{equation}\label{non-degeneracy}
\textup{Ker}(\mathcal{L}_\infty)\cap H_r^2=\{0\}.
\end{equation}
In the first step, by the non-degeneracy, we show invertibility of the operator
$$\mathcal{A}:=\textup{Id}-pu_\infty^{p-1}P_\infty(D)^{-1}.$$

\begin{lem}\label{A invertible}
For each $2\leq q<\infty$, the operator $\mathcal{A}:L_r^2\cap L^q \to L_r^2\cap L^q$ is invertible.
\end{lem}

\begin{proof}
We claim that $pu_\infty^{p-1}P_\infty(D)^{-1}$ is a compact operator on $L_r^2\cap L^q$. Indeed, compactness follows from the well-known localization property of the ground state $u_\infty$ and the compact embedding $H^2(\Omega)\hookrightarrow L^2(\Omega)$ for any bounded set $\Omega$.

If $v\in \textup{Ker}\mathcal{A}$, then $P_\infty(D)^{-1}v\in \textup{Ker}(\mathcal{L}_\infty)$. Hence, it follows from the non-degeneracy \eqref{non-degeneracy} that $P_\infty(D)^{-1}v=0$ and thus $v=0$. Therefore, by the Fredholm alternative, we conclude that $\mathcal{A}$ is invertible.
\end{proof}

By the invertibility of $\mathcal{A}$, we may write
\begin{equation}\label{L_c,infty expansion}
\begin{aligned}
\mathcal{L}_{c;\infty}&=\Big\{\textup{Id}-pu_\infty^{p-1}P_c(D)^{-1}\Big\}P_c(D)\\
&=\Big\{\mathcal{A}+pu_\infty^{p-1}\big(P_\infty(D)^{-1}-P_c(D)^{-1}\big)\Big\}P_c(D)\\
&=\Big\{\textup{Id}+pu_\infty^{p-1}\big(P_\infty(D)^{-1}-P_c(D)^{-1}\big)\mathcal{A}^{-1}\Big\}\mathcal{A}P_c(D).
\end{aligned}
\end{equation}
Thus, the following lemma implies invertibility of $\mathcal{L}_{c;\infty}$.

\begin{lem}\label{invertibility lemma}
Let $2\leq q<\infty$. Suppose that $1<p<\infty$ if $n=1,2$, and that $1<p<\frac{n+2}{n-2}$ if $n\geq 3$. Then, there exists $c_0>0$ such that for $c \geq c_0$,
$$\|pu_\infty^{p-1}(P_\infty(D)^{-1}-P_c(D)^{-1})\mathcal{A}^{-1}\|_{\mathcal{L}(L_r^2\cap L^q)}\leq\frac{1}{2},$$
where $\mathcal{L}(X)=\mathcal{L}(X;X)$.
\end{lem}
\begin{proof}
By a trivial inequality, we write
\begin{equation}\label{eq-7-3}
\begin{split}
&\|pu_\infty^{p-1}(P_\infty(D)^{-1}-P_c(D)^{-1})\mathcal{A}^{-1}\|_{\mathcal{L}(L_r^2\cap L^q)}\\
&\leq p\|u_\infty^{p-1}\|_{\mathcal{L}(L_r^2\cap L^q)}\|P_\infty(D)^{-1}-P_c(D)^{-1}\|_{\mathcal{L}(L_r^2\cap L^q)}\|\mathcal{A}^{-1}\|_{\mathcal{L}(L_r^2\cap L^q)}.
\end{split}
\end{equation}
By H\"older inequality, we have 
$$\|u_\infty^{p-1}\|_{\mathcal{L}(L_r^2\cap L^q)}=\|u_\infty^{p-1}\|_{L^\infty}=\|u_{\infty}\|_{L^{\infty}}^{p-1}.$$ By \eqref{symbol estimate 1}, $\|P_\infty(D)^{-1}-P_c(D)^{-1}\|_{\mathcal{L}(L_r^2\cap L^q)}\leq\frac{C}{c^2}$. Moreover, by Lemma \ref{A invertible}, $\|\mathcal{A}^{-1}\|_{\mathcal{L}(L_r^2\cap L^q)}<\infty$. By inserting these estimates into \eqref{eq-7-3}, we prove the lemma for $c \geq c_0$ with a suitable choice $c_0 >0$.
\end{proof}

\begin{proof}[Proof of Proposition \ref{invertibility}]
By the expression \eqref{L_c,infty expansion} and the above lemmas, we can invert $\mathcal{L}_{c;\infty}$ for $c\geq c_0$, 
$$\mathcal{L}_{c;\infty}^{-1}=P_c(D)^{-1}\mathcal{A}^{-1}\Big\{\textup{Id}+pu_\infty^{p-1}(P_\infty(D)^{-1}-P_c(D)^{-1})\mathcal{A}^{-1}\Big\}^{-1}.$$
Moreover, by the lower bound in Theorem \ref{norm comparability} and Lemma \ref{A invertible} and \ref{invertibility lemma}, we have the bound, 
\begin{align*}
\|\mathcal{L}_{c;\infty}^{-1}\|_{\mathcal{L}(L_r^2\cap L^q; H_r^1\cap W^{1,q})}&\leq\|P_c(D)^{-1}\|_{\mathcal{L}(L_r^2\cap L^q; H_r^1\cap W^{1,q})}\|\mathcal{A}^{-1}\|_{\mathcal{L}(L_r^2\cap L^q)}\\
&\quad\cdot\Big\|\Big\{\textup{Id}+pu_\infty^{p-1}(P_\infty(D)^{-1}-P_c(D)^{-1})\mathcal{A}^{-1}\Big\}^{-1}\Big\|_{\mathcal{L}(L_r^2\cap L^q)}\\
&\leq C,
\end{align*}
where the implicit constant $C$ is independent of the choice of $c\geq c_0$.
\end{proof}

\subsection{First term bound in \eqref{contraction mapping}}
We now prove that we can make the term $\mathcal{R}_c=(\mathcal{L}_{c;\infty})^{-1}(P_\infty(D)-P_c(D))u_\infty$ arbitrarily small choosing large $c\geq 1$.

\begin{lem}[First term bound]\label{lem-2-2}
Let $2\leq q<\infty$.  Then, we have
$$
\|\mathcal{R}_c \|_{H_r^1\cap W^{1,q}}=\left\{\begin{aligned}
&O \left(\frac{1}{c}\right)&&\textup{ if }1<p\leq 2,\\
&O \left(\frac{1}{c^2}\right)&&\textup{ if }p>2.
\end{aligned}\right.$$
\end{lem}

\begin{proof}
We recall that the non-relativistic ground state $u_\infty$ is contained in $H_r^{2+\lfloor p\rfloor}\cap W_r^{2+\lfloor p\rfloor, q}$ for all $q$, where $\lfloor p\rfloor$ is the largest integer less than or equal to $p$.  Thus, if $p>2$, then by Proposition \ref{invertibility}, \eqref{symbol estimate 1} and the upper bound in Theorem \ref{norm comparability}, we get
\begin{align*}
\|\mathcal{R}_c \|_{H_r^1\cap W^{1,q}}&\leq\|(\mathcal{L}_{c;\infty})^{-1}\|_{\mathcal{L}(L_r^2\cap L^q;H_r^1\cap W^{1,q})}\left\|\frac{P_\infty(D)-P_c(D)}{P_\infty(D)P_c(D)}\right\|_{\mathcal{L}(L_r^2\cap L^q)}\|P_\infty(D)P_c(D)u_\infty\|_{L_r^2\cap L^q}\\
&\leq \frac{C}{c^2} \|P_\infty(D)u_\infty\|_{H_r^2\cap W_r^{2,q}}\leq \frac{C}{c^2}\|u_\infty\|_{H_r^4\cap W_r^{4,q}}.
\end{align*}
Similarly, if $1<p<2$, then by Proposition \ref{invertibility}, \eqref{symbol estimate 2} and the upper bound in Theorem \ref{norm comparability}, we obtain
\begin{align*}
\|\mathcal{R}_c \|_{H_r^1\cap W^{1,q}}&\leq\|(\mathcal{L}_{c;\infty})^{-1}\|_{\mathcal{L}(L_r^2\cap L^q;H_r^1\cap W^{1,q})}\left\|\frac{P_\infty(D)-P_c(D)}{P_\infty(D)^{\frac{1}{2}}P_c(D)}\right\|_{\mathcal{L}(L_r^2\cap L^q)}\|P_\infty(D)^{\frac{1}{2}}P_c(D)u_\infty\|_{L_r^2\cap L^q}\\
&\leq \frac{C}{c} \|P_\infty(D)^{\frac{1}{2}}u_\infty\|_{H_r^2\cap W_r^{2,q}}\leq \frac{C}{c}\|u_\infty\|_{H_r^3\cap W_r^{3,q}}.
\end{align*}
\end{proof}

\subsection{Nonlinear term bounds in \eqref{contraction mapping}}\label{sec: nonlinear estimates}

Next, we establish the estimates for the nonlinear term $\mathcal{L}_{c;\infty}^{-1}\mathcal{Q}(w)$ for small $w$.
\begin{prop}[Nonlinear estimates]\label{H^1 nonlinear estimate}
Fix any $q>n$ and suppose that $0<\delta\leq\|u_\infty\|_{H^1}$. Then for $c\geq c_0$, where $c_0\geq 1$ is a large number from Proposition \ref{invertibility}, we have
\begin{align}\
\|\mathcal{L}_{c;\infty}^{-1}\mathcal{Q}(w)\|_{H_r^1\cap W^{1,q}}&\leq C\delta^{\min\{p,2\}},\label{eq: H^1 estimate}\\
\|\mathcal{L}_{c;\infty}^{-1}\mathcal{Q}(w)-\mathcal{L}_{c;\infty}^{-1}\mathcal{Q}(\tilde{w})\|_{H_r^1\cap W^{1,q}}&\leq C\delta^{\min\{p-1,1\}}\|w-\tilde{w}\|_{H_r^1\cap W^{1,q}}\label{eq: H^1 difference estimate}
\end{align}
for any $w,\tilde{w}\in H_r^1\cap W^{1,q}$ with $\|w\|_{H_r^1\cap W^{1,q}},\|\tilde{w}\|_{H_r^1\cap W^{1,q}}\leq\delta$.
\end{prop}

\begin{proof}
It suffices to show the second inequality \eqref{eq: H^1 difference estimate} in the proposition, since the former inequality follows from the latter with $\tilde{w}=0$. 

By the definition \eqref{nonlinearity} and the fundamental theorem of calculus, we write
\begin{align*}
\mathcal{Q}(w)-\mathcal{Q}(\tilde{w})&=\Big\{|u_\infty+w|^{p-1}(u_\infty+w)-|u_\infty+\tilde{w}|^{p-1}(u_\infty+\tilde{w})\Big\}\\
&\quad\quad-p u_\infty^{p-1}(w-\tilde{w})\\
&=\int_0^1 \frac{d}{dt}\Big[|u_\infty+(1-t)\tilde{w}+tw|^{p-1}(u_\infty+(1-t)\tilde{w}+tw)\Big]dt\\
&\quad\quad-p u_\infty^{p-1}(w-\tilde{w})\\
&=p\int_0^1\big(|u_\infty+(1-t)\tilde{w}+tw|^{p-1}-u_\infty^{p-1}\big)(w-\tilde{w})dt.
\end{align*}
Suppose that $1<p\leq2$. Then, by the elementary inequality
$$||a|^\ell-|b|^\ell|\leq \big||a|-|b|\big|^\ell\leq |a-b|^\ell\quad\textup{ if }0<\ell<1,$$
we have
$$|\mathcal{Q}(w)-\mathcal{Q}(\tilde{w})|\leq C(|w|+|\tilde{w}|)^{p-1}|w-\tilde{w}|.$$
Thus, by Proposition \ref{invertibility} and the H\"older inequality and the Sobolev inequality $H_r^1\cap W^{1,q}\hookrightarrow L_r^2\cap L_r^\infty$, we prove that
\begin{equation}\label{nonlinear estimate proof 1}
\begin{aligned}
\|\mathcal{L}_{c;\infty}^{-1}\mathcal{Q}(w)-\mathcal{L}_{c;\infty}^{-1}\mathcal{Q}(\tilde{w})\|_{H_r^1\cap W^{1,q}}&\leq C\|\mathcal{Q}(w)-\mathcal{Q}(\tilde{w})\|_{L_r^2\cap L_r^{q}}\\
&\leq C\big(\|w\|_{H_r^1\cap W^{1,q}}+\|\tilde{w}\|_{H_r^1\cap W^{1,q}}\big)^{p-1}\|w-\tilde{w}\|_{H_r^1\cap W^{1,q}}.
\end{aligned}
\end{equation}
If $p>2$, using the fundamental theorem of calculus again, we find
$$|\mathcal{Q}(w)-\mathcal{Q}(\tilde{w})|\leq C(u_\infty+|w|+|\tilde{w}|)^{p-2}(|w|+|\tilde{w}|)|w-\tilde{w}|.$$
From this and estimating as above, we prove that
\begin{equation}\label{nonlinear estimate proof 2}
\begin{aligned}
\|\mathcal{L}_{c;\infty}^{-1}\mathcal{Q}(w)-\mathcal{L}_{c;\infty}^{-1}\mathcal{Q}(\tilde{w})\|_{H_r^1\cap W^{1,q}}&\leq C\|\mathcal{Q}(w)-\mathcal{Q}(\tilde{w})\|_{L_r^2\cap L_r^{q}}\\
&\leq C\|(u_\infty+|w|+|\tilde{w}|)^{p-2}(|w|+|\tilde{w}|)|w-\tilde{w}|\|_{L_r^2\cap L_r^{q}}\\
&\leq C\big(\|u_\infty\|_{H_r^1\cap W^{1,q}}+\|w\|_{H_r^1\cap W^{1,q}}+\|\tilde{w}\|_{H_r^1\cap W^{1,q}}\big)^{p-2}\\
&\quad\quad\cdot\big(\|w\|_{H_r^1\cap W^{1,q}}+\|\tilde{w}\|_{H_r^1\cap W^{1,q}}\big)\|w-\tilde{w}\|_{H_r^1\cap W^{1,q}}.
\end{aligned}
\end{equation}
Thus, the proposition is proved.
\end{proof}

\begin{rem}\label{Why Hormander-Mikhlin}
As mention in the introduction, if we only use the $L^2$-boundedness in Theorem \ref{thm difference} and \ref{norm comparability} based on the point-wise bounds on the symbols, then we can close the estimates \eqref{nonlinear estimate proof 1} and \eqref{nonlinear estimate proof 2} with $q=2$ only when $1<p\leq\frac{n}{n-2}$. The symbolic analysis in Section 2 allows us to employ the full range of the $W_r^{1,q}$-Sobolev norms in \eqref{nonlinear estimate proof 1} and \eqref{nonlinear estimate proof 2}, thus we can cover the full range of $p$, i.e., $1<p<\frac{n+2}{n-2}$, such that the ground state $u_\infty$ exists.
\end{rem}
\subsection{Construction of a solution to \eqref{eq w''}}\label{sec: construction of w}

Now we are ready to construct a solution to the equation \eqref{eq w''} near the ground state $u_\infty$.

\begin{prop}[Existence of a fixed point for $\Phi_c$]\label{existence and uniqueness}\mbox{~}
Fix any $q>n$. Then, given a sufficiently small $\delta>0$, there exists $c_0>0$ such that if $c \geq c_0$, then $\Phi_c$ has a unique fixed point in
$$B_\delta^1:= \left\{ w \in H_r^1\cap W^{1,q} : \|w\|_{H_r^1\cap W^{1,q}} \leq \delta\right\}.$$
\end{prop}

\begin{proof}
First, by Lemma \ref{lem-2-2}, we choose large $c_0\geq 1$ such that $\|\mathcal{R}_c \|_{H_r^1\cap W^{1,q}}\leq\frac{\delta}{2}$ for all $c\geq c_0$. Hence, it follows from Proposition \ref{H^1 nonlinear estimate} that if $w,\tilde{w}\in B_\delta^1$, then $\|\Phi_c(w)\|_{H_r^1\cap W^{1,q}}\leq\frac{\delta}{2}+C\delta^{\min\{p,2\}}\leq \delta$ and  $\|\Phi_c(w)-\Phi_c(\tilde{w})\|_{H_r^1\cap W^{1,q}}\leq C\delta^{\min\{p-1,1\}}\|w-\tilde{w}\|_{H_r^1\cap W^{1,q}}\leq\frac{1}{2}\|w-\tilde{w}\|_{H_r^1\cap W^{1,q}}$, provided that $C^{\min\{p-1,1\}}\delta\leq\frac{1}{2}$. Thus, we conclude that $\Phi_c$ has a unique fixed point in $B_\delta^1$.
\end{proof}

\subsection{Construction of a solution $u_c$ to \eqref{reduced eq-main}}\label{sec: properties}
We prove that $u_c=u_\infty+w$, where $w$ is given in Proposition \ref{existence and uniqueness}, is indeed a solution to the pseudo-relativistic Schr\"odinger equation \eqref{reduced eq-main}.
\begin{lem}[Construction of a solution to \eqref{reduced eq-main}]\label{construction of u_c}
$w$ is a solution to the equation \eqref{eq w''}, i.e., $w=\Phi_c(w)$ in $H_r^1\cap W^{1,q}$ for some $q>n$ if and only if $u_c=u_\infty+w$ solves \eqref{reduced eq-main}.
\end{lem}

\begin{proof}
It is proved in Proposition \ref{existence and uniqueness} that  $w=\mathcal{R}_c+(\mathcal{L}_{c;\infty})^{-1}\mathcal{Q}(w)$ in $H_r^1\cap W^{1,q}$, and so
\begin{align*}
0&=\mathcal{L}_{c;\infty}w-\mathcal{L}_{c;\infty}\mathcal{R}_c-\mathcal{Q}(w)\\
&=(P_c(D)-pu_\infty^{p-1})w-(P_\infty(D)-P_c(D))u_\infty-\big(|u_c|^{p-1}u_c-u_\infty^p-pu_\infty^{p-1}w\big)\\
&=P_c(D)w-pu_\infty^{p-1}w-P_\infty(D)u_\infty+P_c(D)u_\infty-|u_c|^{p-1}u_c+u_\infty^p+pu_\infty^{p-1}w\\
&=P_c(D)u_c-|u_c|^{p-1}u_c.
\end{align*}
Thus $u_c$ is a solution to \eqref{reduced eq-main}.
\end{proof}

\section{Non-existence result}\label{sec-4}

We prove our non-existence theorem (Theorem \ref{thm-3}). By scaling (see Section \ref{algebra}), we may take $m=\frac{1}{2}$ and $\mu = 1$. It suffices to show non-existence assuming that $1 \geq \frac{c^2}{2}$ and $p \geq \frac{n+1}{n-1}$, or that $1 < \frac{c^2}{2}$ and $p \geq \frac{n+2}{n-2}$.

Let $u_c \in H^{1/2}(\R^n) \cap L^\infty(\R^n)$ a solution to the pseudo-relativistic NLS \eqref{eq-main}, which will be shown to be zero in the end. Then, it has a unique extension $U(x,t) \in H^1(\R^{n+1}_+)$ to the upper half-plane $\R^{n+1}_+=\{(x,t): x\in\mathbb{R}^n \textup{ and }t>0\}$ such that 
\begin{equation}\label{R NLS ext}
\left\{\begin{aligned}
\left(-c^2\Delta_{(x,t)} +\frac{c^4}{4}\right)U(x,t) &= 0&& \text{in } \R^{n+1}_+, \\
 U(x,0) &= u_c(x) &&\text{in } \R^n
\end{aligned}\right.
\end{equation}
and
\[-c\frac{\partial}{\partial t}U(x,0) = \sqrt{-c^2\Delta_x+\frac{c^4}{4}}\, u_c(x)\]
in distribution sense, which immediately implies that
\[-c\frac{\partial}{\partial t}U(x,0) = \left(\frac{c^2}{2}-1\right)U(x,0) +|U|^{p-1}U(x,0)\]
because $u_c(x)=U(x,0)$ solves \eqref{eq-main} (see \cite{CS, CN2}). Since $u_c \in L^\infty(\R^n)$, by the maximum principle, we have $U \in L^\infty(\R^{n+1}_+)$. 
Then, it follows from the standard elliptic regularity estimates that $U \in C^{\alpha}(\overline{\R^{n+1}_+})$ for some $\alpha > 0$. In particular, $U$ is continuous up to the boundary $\partial\R^{n+1}_+ = \R^n$. Moreover, the extension $U$ satisfies the Pohozaev-type identities.
\begin{lem}Let $U \in H^{1/2} (\mathbb{R}^{n+1}_{+})$ be a solution to \eqref{R NLS ext}. Then we have the following identities.
\begin{equation}\label{Nehari}
\begin{aligned}
&\int_{\R^{n+1}_+}c^2|\nabla U(x,t)|^2dxdt +\int_{\R^{n+1}_+}\frac{c^4}{4}|U(x,t)|^2\,dxdt\\
&=\left(\frac{c^2}{2}-1\right)\int_{\R^n}c|U(x,0)|^2\,dx+\int_{\R^n}c|U(x,0)|^{p+1}\,dx,
\end{aligned}
\end{equation}
\begin{equation}\label{Poho1}
\begin{aligned}
&\frac{n-1}{2}\int_{\R^{n+1}_+}c^2|\nabla U(x,t)|^2dxdt +\frac{n+1}{2}\int_{\R^{n+1}_+}\frac{c^4}{4}|U(x,t)|^2\,dxdt\\
&=\left(\frac{c^2}{2}-1\right)\frac{n}{2}\int_{\R^n}c|U(x,0)|^2\,dx+\frac{n}{p+1}\int_{\R^n}c|U(x,0)|^{p+1}\,dx,
\end{aligned}
\end{equation}
and
\begin{equation}\label{Poho2}
\begin{aligned}
&\frac{n-2}{2}\int_{\R^{n+1}_+}c^2|\nabla_x U(x,t)|^2dxdt +\frac{n}{2}\int_{\R^{n+1}_+}c^2|\partial_t U(x,t)|^2dxdt +\frac{n}{2}\int_{\R^{n+1}_+}\frac{c^4}{4}|U(x,t)|^2\,dxdt\\
&=\left(\frac{c^2}{2}-1\right)\frac{n}{2}\int_{\R^n}c|U(x,0)|^2\,dx+\frac{n}{p+1}\int_{\R^n}c|U(x,0)|^{p+1}\,dx.
\end{aligned}
\end{equation}
\end{lem}
\begin{proof}
We multiply \eqref{R NLS ext} by the three test functions $U(x,t)$, $(x,t)\cdot\nabla U(x,t)$ and $x\cdot\nabla_x U(x,t)$, and integrate on the $(n+1)$-dimensional upper half ball $B^{n+1}_+(0,R)$ of radius $R > 0$ and centered at $0$.
Then, by integration by parts, we get the above three integral identities, after taking the limit under a suitable choice of the sequence $R_j \to \infty$. We omit the details of this procedure, because it is quite standard in the literature.
Here, we note that continuity of $U$ is required to guarantee that the boundary integral terms, which appear
whenever we do integration by parts, are well defined. 
\end{proof}

We also recall the following trace inequality.
\begin{lem}\label{trace-ineq}
For every $U \in H^1(\R^{n+1}_+)$, we have
\[
\int_{\R^n}|U(x,0)|^2\,dx \leq 2\|U\|_{L^2(\R^{n+1}_+)}\|\partial_tU\|_{L^2(\R^{n+1}_+)}.
\]
\end{lem}
\begin{proof}
From the density argument, we may assume that $U \in C^\infty_c(\overline{\R^{n+1}_+})$. 
Observe
\[
|U(x,0)|^2= -\int_0^\infty\partial_t\left(|U(x,t)|^2\right)\,dt \leq 2\int_0^\infty|U(x,t)||\partial_tU(x,t)|\,dt.
\]
Then one has from the H\"older inequality
\[\int_{\R^n}|U(x,0)|^2\,dx \leq 2\int_{\R^n}\int_0^\infty|U(x,t)||\partial_tU(x,t)|\,dtdx
\leq 2\|U\|_{L^2(\R^{n+1}_+)}\|\partial_tU\|_{L^2(\R^{n+1}_+)}.
\]
\end{proof}

Using the Pohozaev-type identities and the trace inequality, we prove non-existence.

\begin{proof}[Proof of Theorem \ref{thm-3}]
Suppose that $1 \geq \frac{c^2}{2}$ and $p \geq \frac{n+1}{n-1}$. Then, combining \eqref{Nehari} and \eqref{Poho1} to cancel out the last term, we obtain that
\begin{equation}\label{equ1}
\begin{aligned}
&\left(\frac{n-1}{2}-\frac{n}{p+1}\right)\int_{\R^{n+1}_+}c^2|\nabla U(x,t)|^2dxdt +\left(\frac{n+1}{2}-\frac{n}{p+1}\right)\int_{\R^{n+1}_+}\frac{c^4}{4}|U(x,t)|^2\,dxdt\\
&=\left(\frac{c^2}{2}-1\right)\left(\frac{n}{2}-\frac{n}{p+1}\right)\int_{\R^n}c|U(x,0)|^2\,dx.
\end{aligned}
\end{equation}
Thus, it follows that
\begin{equation*}
\int_{\R^{n+1}_+}|U(x,t)|^2\,dxdt=0,
\end{equation*}
because $\frac{n-1}{2}-\frac{n}{p+1}\geq 0$, $\frac{n+1}{2}-\frac{n}{p+1}>0$ and $(\frac{c^2}{2}-1)(\frac{n}{2}-\frac{n}{p+1})\leq 0$. Consequently, $u_c(x) = U(x,0)$ is identically zero by the continuity of $U$.

Now we assume that $1 < \frac{c^2}{2}$ and $p \geq \frac{n+2}{n-2}$. Then, combining \eqref{Nehari} and \eqref{Poho2}, we obtain that
\begin{equation}\label{equ2}
\begin{split}
&\left(\frac{n-2}{2}-\frac{n}{p+1}\right)\int_{\R^{n+1}_+}c^2|\nabla_x U(x,t)|^2\,dxdt
+\left(\frac{n}{2}-\frac{n}{p+1}\right)\int_{\R^{n+1}_+}c^2|\partial_tU(x,t)|^2\,dxdt \\
&\qquad+\left(\frac{n}{2}-\frac{n}{p+1}\right)\int_{\R^{n+1}_+}\frac{c^4}{4}|U(x,t)|^2\,dxdt
\\
&=\left(\frac{c^2}{2}-1\right)\left(\frac{n}{2}-\frac{n}{p+1}\right)\int_{\R^n}c|U(x,0)|^2\,dx.
\end{split}
\end{equation}
We see from Lemma \ref{trace-ineq} and Young's inequality that
\[
\begin{aligned}
\left(\frac{c^2}{2}-1\right)\int_{\R^n}c|U(x,0)|^2\,dx &\leq 
2\left(\frac{c^2}{2}-1\right)\|U\|_{L^2(\R^{n+1}_+)}c\|\partial_tU\|_{L^2(\R^{n+1}_+)} \\
& \leq \left(\frac{c^2}{2}-1\right)^2\int_{\R^{n+1}_+}|U(x,t)|^2\,dxdt + \int_{\R^{n+1}_+}c^2|\partial_tU(x,t)|^2\,dxdt.
\end{aligned}
\]
Inserting this to \eqref{equ2}, 
\begin{equation*}
\begin{split}
&\left(\frac{n-2}{2}-\frac{n}{p+1}\right)\int_{\R^{n+1}_+}c^2|\nabla_x U(x,t)|^2\,dxdt+\left(\frac{n}{2}-\frac{n}{p+1}\right)\int_{\R^{n+1}_+}\frac{c^4}{4} |U(x,t)|^2\,dxdt
\\
& \leq \left(\frac{n}{2}-\frac{n}{p+1}\right)\left(\frac{c^2}{2}-1\right)^2\int_{\R^{n+1}_+}|U(x,t)|^2\,dxdt.\\
\end{split}
\end{equation*}
From this, by the assumption, we finally deduce
\[
(c^2-1)\int_{\R^{n+1}_+}|U(x,t)|^2\,dxdt = 0.
\]
This again implies that $u_c(x)$ is identically zero. 
\end{proof}

\section{Concluding remarks}\label{sec-5}

In this section, we present some properties of the solution $u_c$ to the pseudo-relativistic NLS \eqref{eq-main}, constructed in Section \ref{sec: existence result}. Throughout this section, we assume that
\begin{align*}
\left\{
\begin{aligned}
&1 < p<\infty&&\textup{if }n=1,2,\\
&1 < p<\tfrac{n+2}{n-2}&&\textup{if }n\geq 3,
\end{aligned}
\right.
\end{align*}
and that $mc^2 \geq \kappa_0\mu$, where $\kappa_0=c_0^2$ and $c_0 \geq 1$ is a large constant given in in Section \ref{sec: existence result}.

First, we prove uniqueness of a solution to \eqref{eq-main} among radially symmetric functions near the ground state $u_\infty$ to the non-relativistic NLS \eqref{non-relativistic eq}.

\begin{prop}\label{uniqueness-small ball}
There exists some $\delta > 0$ such that a solution to \eqref{eq-main} is unique in
$$B_\delta(u_\infty) := \{ u \in H^1_r \cap L^\infty:\  \|u-u_\infty\|_{H^1\cap L^\infty} < \delta \}.$$
\end{prop}
\begin{proof}
Let $u_c \in H^1_r \cap L^\infty$ be a solution to \eqref{eq-main}. Then by the argument in Lemma \ref{construction of u_c}, $w := u_c -u_\infty$ is a fixed point of the map $\Phi_c$. 
Then $w$ is a unique fixed point in a small ball $B_\delta(0)$ because $\Phi_c$ is a contraction map. 
Note that two norms $\|\cdot\|_{H^1\cap L^\infty}$ and $\|\cdot\|_{H^1\cap W^{1,q}}$ are equivalent since we assume $q > n$. 
This shows uniqueness of $u_c$ in a small ball $B_\delta(u_\infty)$. 
\end{proof}

We also obtain the rate of convergence for the non-relativistic limit $u_c\to u_\infty$.
\begin{prop}[Rate of convergence]\label{rate of convergence}
Let $u_c$ be the solution to \eqref{eq-main} constructed in Section \ref{sec: existence result}. 
Then, for any $q>n$, we have
$$\|u_c-u_\infty\|_{H^1\cap W^{1,q}}=\left\{\begin{aligned}
&O \left(\frac{1}{c}\right)&&\textup{ if }1<p\leq 2,\\
&O \left(\frac{1}{c^2}\right)&&\textup{ if }p>2.
\end{aligned}\right.$$
\end{prop}

\begin{proof}
By Lemma \ref{lem-2-2}, we may choose $\delta=\frac{A}{c^a}$ such that $\|\mathcal{R}_c\|_{H_r^1\cap W^{1,q}}\leq\delta$, where $A>0$ is some large number and $a=1$ or $2$ depending on the rate in Lemma \ref{lem-2-2}. Then, repeating the proof of Proposition \ref{existence and uniqueness}, one can show that $\Phi_c$ is contractive on the $\frac{A}{c^a}$-ball for sufficiently large $c$. Let $\tilde{w}$ be the fixed point in the $\frac{A}{c^a}$-ball. Then, by uniqueness, the solution $\tilde{w}$ equals to the solution $w=u_c-u_\infty$ in Proposition \ref{existence and uniqueness}. Therefore, we conclude that the difference $u_c-u_\infty=\tilde{w}$ is in the ball of radius $\frac{A}{c^a}$.
\end{proof}

Combining the above two propositions, we conclude that the solution $u_c$, in Section \ref{sec: existence result}, is the only radially symmetric real-valued solution to the pseudo-relativistic NLS \eqref{eq-main} converging to the non-relativistic ground state $u_\infty$.
\begin{cor}\label{uniqueness-rate}
Let $\{u_c\}$ be a sequence of  solutions to \eqref{eq-main} in $H^1_r \cap L^\infty$ such that it converges to $u_\infty$ in $H^1_r \cap L^\infty$ as $c \to \infty$.
Then for sufficiently large $c\geq 1$, $u_c$ is unique, and it converges with the rate given in Proposition \ref{rate of convergence}. 
\end{cor}
In \cite{CHS, CS}, the authors prove that in the $H^{1/2}$-subcritical range $1 < p < \tfrac{n+1}{n-1}$, a positive radial ground state to \eqref{eq-main} belongs to $H^1 \cap L^{\infty}$ and converges to $u_\infty$
so, by the uniqueness, our solution $u_c$ in this case is the same as the ground state to \eqref{eq-main} for large $c$.
By a \textit{ground state}, we mean a solution to \eqref{eq-main} which attains the minimum value of an associated functional $I_{c}$ among all nontrivial solutions, where
\begin{equation}\label{action functional}
I_{c}(u) = \frac12\int_{\R^n}\left( \sqrt{-c^2\Delta + m^2 c^4}-mc^2\right)u\overline{u} +\mu |u|^2\,dx -\frac{1}{p+1}\int_{\R^n}|u|^{p+1}\,dx
\end{equation}
Obviously, $u$ solves \eqref{eq-main} if and only if it is a critical point of the functional $I_{c}$, because \eqref{eq-main} is its Euler-Lagrange equation. 
Thus, a ground state $u_c$ can be rephrased as a critical point of the functional $I_{c}$ that minimizes the value of $I_{c}$ among all nontrivial critical points, that is,
\begin{equation}\label{minimization problem}
I_{c}(u_c)=\min_{v\in H^{1/2}}\Big\{I_{c}(v) ~|~ v \neq 0,\, I_{c}'(v) = 0\Big\}.
\end{equation}
Thus, Corollary \ref{uniqueness-rate} gives an alternative proof of uniqueness of a radially symmetric non-negative ground state to the $H^{1/2}$-subcritical pseudo-relativistic NLS \eqref{eq-main}.
\begin{cor}
If $1 < p <\frac{n+1}{n-1}$, then a positive radial ground state to \eqref{eq-main} is unique for sufficiently large $c$. 
\end{cor}

We finally remark that a ground state, in the sense of \eqref{minimization problem}, is well-defined only when the nonlinearity is $H^{1/2}$-subcritical or critical, i.e., $1 < p \leq \frac{n+1}{n-1}$, since by the Sobolev embedding $H^{1/2}\hookrightarrow L^{p+1}$, the functional $I_{c}$ is well-defined and continuously differentiable on $H^{1/2}$ in this case. 
However, if we define a ground state (in a weak sense) as a minimizer of the action functional $I_c$ among all nontrivial solutions $v \in H^1\cap L^\infty$ to \eqref{eq-main},
then the meaning of a ground state may make sense even in the $H^{1/2}$-supercritical case $\frac{n+1}{n-1}<p < \frac{n+2}{n-2}$.
We strongly speculate that our solution $u_c$, constructed in Theorem \ref{existence}, would be a ground state in this sense.  This seems to be an interesting open question worth to be answered in a future work.


\begin{thebibliography}{20}



\bibitem{BL} H. Berestycki, P.-L. Lions, Nonlinear scalar field equations. I. Existence of a ground state. Arch. Rational Mech. Anal. \textbf{82}  (1983), no. 4, 313--345.

\bibitem{CHS} W. Choi, Y. Hong, J. Seok, Optimal convergence rate of nonrelativistic limit for the nonlinear pseudo-relativistic equations, arXiv:1610.06030. 

\bibitem{CS} W. Choi, J. Seok, Nonrelativistic limit of standing waves for pseudo-relativistic nonlinear Schr\"odinger equations. J. Math. Phys. \textbf{57} (2016), no. 2, 021510, 15 pp.


\bibitem{CiS} S. Cingolani, S. Secchi, Ground states for the pseudo-relativistic Hartree equation with external potential, in press on Proceedings of the Royal Society of Edinburgh.

\bibitem{CiS2} S. Cingolani, S. Secchi, Semiclassical analysis for pseudo-relativistic Hartree equations. J. Differential Equations  \textbf{258}  (2015),  no. 12, 4156–4179. 

\bibitem{CGG} A. Comech, M. Guan, and S. Gustafson, On linear instability of solitary waves for the nonlinear Dirac equation. Ann. Inst. H. Poincaré Anal. Non Lin\'eaire  \textbf{31}  (2014), 639--654. 

\bibitem{CN1} V. Coti-Zelati, M. Nolasco, Existence of ground states for nonlinear, pseudo-relativistic Schr\"odinger equations. Atti Accad. Naz. Lincei Cl. Sci. Fis. Mat. Natur. Rend. Lincei (9) Mat. Appl. \textbf{22} (2011), 51--72.

\bibitem{CN2} V. Coti-Zelati, M. Nolasco, Ground states for pseudo-relativistic Hartree equations of critical type. Rev. Mat. Iberoam.  \textbf{29}  (2013), 1421--1436.







\bibitem{FLe} R. L. Frank, E. Lenzmann, Uniqueness of non-linear ground states for fractional Laplacians in $\R$. Acta Math. {\bf 210} (2013), no. 2, 261--318.

\bibitem{FLS} R. L. Frank, E. Lenzmann, L. Silvestre, Uniqueness of radial solutions for the fractional Laplacian. Comm. Pure Appl. Math. {\bf 69} (2016), no. 9, 1671--1726.


\bibitem{Gr} L. Grafakos, Classical Fourier Analysis (2nd ed.), Springer (2008).

\bibitem{G} M. Guan, Solitary wave solutions for the nonlinear Dirac equations, arXiv:0812.2273.








\bibitem{Kw} M. K. Kwong, Uniqueness of positive solutions of $\Delta u - u + u^p =0$ in $\mathbb{R}^n$.  Arch. Rational Mech. Anal. \textbf{105} (1989), 243--266.




\bibitem{L2} E. Lenzmann, Uniqueness of ground states for pseudo-relativistic Hartree equations. Anal. PDE \textbf{2} (2009),  1--27.


\bibitem{MMP} A. M. Micheletti, M. Musso, and A. Pistoia, Super-position of spikes for a slightly supercritical elliptic equation in $\mathbb{R}^N$. Discrete Contin. Dyn. Syst. \textbf{12}  (2005),  747--760. 














\bibitem{M} D. Mugnai, Pseudorelativistic Hartree equation with general nonlinearity: existence, non-existence and variational identities, Adv. Nonlinear Stud.  \textbf{13}  (2013),  799--823.





\bibitem{O} H. Ounaies, Perturbation method for a class of nonlinear Dirac equations, Differential Integral Equations, \textbf{13}  (2000),  707--720. 


\bibitem{S} W. Strauss, Existence of solitary waves in higher dimensions. Comm. Math. Phys. \textbf{55}  (1977), no. 2, 149--162.

\bibitem{TWY} J. Tan, Y. Wang, J. Yang, Nonlinear fractional field equations. Nonlinear Anal. \textbf{75} (2012), 2098--2110.


\end{thebibliography}
\end{document}